\documentclass[a4paper,11pt,draft]{article}
\usepackage{amstext}
\usepackage{amsmath}
\usepackage{amssymb}
\usepackage{amsfonts}
\usepackage{amsthm}
\usepackage{parskip}

\usepackage{graphicx}
\usepackage{latexsym}
\usepackage[active]{srcltx}

\usepackage{epsfig}
\usepackage{tikz}
\usetikzlibrary{matrix,arrows,decorations.pathmorphing}
\usepackage{anysize}
\marginsize{2,5cm}{2,5cm}{2cm}{3cm}
\input xy
\xyoption{all}
\begin{document}

\newtheoremstyle{basico}% name of the style to be used
  {0,2cm}% measure of space to leave above the theorem. E.g.: 3pt
  {}% measure of space to leave below the theorem. E.g.: 3pt
  {\itshape}% name of font to use in the body of the theorem
  {0,5cm}% measure of space to indent
  {\bfseries}% name of head font
  {}% punctuation between head and body
  {0,2cm}% space after theorem head; " " = normal interword space
  {\thmname{#1}\thmnumber{ #2}:\thmnote{ #3}}% Manually specify head
\theoremstyle{basico}  
  
\newtheorem{teo}{Theorem}[section]
\newtheorem{coro}[teo]{Corollary}
\newtheorem{lema}[teo]{Lemma}
\newtheorem{prop}[teo]{Propostion}
\newtheorem*{defn}{Definition}
\newtheorem{obs}[equation]{Observation}
\newtheorem{ejem}[equation]{Example}
\newtheorem{obsnot}{Notation}
\newtheorem{afim}{Afirmation}
\newtheorem{teoprin}{Theorem}

\newcommand{\noi}{\noindent}

\newcommand{\La}{\Lambda}
\newcommand{\s}{\sigma}
\newcommand{\al}{\alpha}
\newcommand{\de}{\delta}
\newcommand{\e}{\epsilon}
\newcommand{\ga}{\gamma}
\newcommand{\be}{\beta}

\newcommand{\la}{\Lambda}
\newcommand{\N}{\mathbb{N}}
\newcommand{\Z}{\mathbb{Z}}
\newcommand{\R}{\mathbb{R}}
\newcommand{\A}{\mathcal{A}}
\newcommand{\B}{\mathcal{B}}
\newcommand{\U}{\mathcal{U}}

\title{H\"older stability for $C^r$ central translations}
\date{\today}
\author{Javier Correa\\
Enrique R. Pujals}
\maketitle

\abstract{We consider the class of diffeomorphisms of a manifold that its differential keeps invariant  a one-dimensional subbundle $E$. For that type of diffeomorphisms is naturally defined a one-parameter family called $E-$translation. We prove that if a diffeomorphisms in above mentioned class is conjugate to its $E-$translation and the conjugacy is at distance $\alpha$-H\"older to the identity respect to the parameter and   $\alpha>1/2$, then the   $E$-direction is hyperbolic.

 This theorem is also sharp as it is be discussed with some examples. We also deal with the continuously stable case in the Skew-Products context with one-dimensional fibers, requiring extra hypothesis  along the fibers   like either non-negative second derivative or negative Schwartzian.}

\setlength{\parskip}{0,2cm}
\setlength{\parindent}{0.5cm}

%%%%%%%%%%%%%%%%%%%%%%%%%%%%%%%%%%%%%%%%%%%%%%%%%%%%%%%%%%%%%%%%%%%INTRODUCCION%%%%%%%%%%%%%%%%%%%%%%%%%%%%%%%%%%%%%%%%%%%%%%%%%%

\section{Introduction}
One goal in Dynamical Systems is to describe open sets in the space of diffeomorphisms of a manifold. One of that opens space is characterized by the  stable ones. Given  a $C^r$ Riemannian manifold $M$ and  a diffeomorphisms $f:M\to M$, it is said that $f$ is $C^r$ structural stable, if there exists a neighborhood $\U(f)$ of $f$ in the $C^r$-topology such that for every $g\in \U(f)$ there is  an homeomorphism $\varphi:M\to M$ verifying $\varphi \circ f = g \circ \varphi$. If it is required that the conjugacy only holds between the corresponding non-wandering sets, it is said that $f$ is $\Omega-$stable. In the study of such a question S. Smale introduced the concept of hyperbolicity in \cite{Sm67}. J. Palis and S. Smale conjectured in \cite{PaSm68} that a diffeomorphisms is  $C^r$ structural stable  if and only if  it is Axiom A (the non-wandering set is hyperbolic and the periodic points are dense) with strong transversality and it is $ C^r$ $\Omega-$stable diffeomorphisms if and only if it is an  Axiom A with no cycles. Both conjectures were proved in the $C^1$ topology; for the first one (structural stability) the converse was proved by Robinson in \cite{Ro76} and the direct part by R. Ma\~n\'e in  \cite{Ma88}; for the second one ($\Omega-$stability) the direct was proved by J. Palis in \cite{Pa88} and the converse by S. Smale in \cite{Sm70}. The direct part remains wide open in the the $C^r$ topology for $r>1$ (see \cite{Pu08}). One way to address this problem,  is by assuming  a stronger notion of stability which  requires some regularity in the variation of the conjugacy $\varphi$ according to the variation of $g$. More precisely, it is said that $f$ is absolute stable if there exist $C>0$ such that $d(\varphi, Id) \leq C d(f,g)$. Using implicit function arguments,  J. Robbin in \cite{Ro71}  concluded that Axiom A plus strong transversality implies absolute stability. The converse was proved in the $C^1$ topology by J. Franks in \cite{Fr72} and J. Guckenheimer in \cite{Gu72}. In the $C^r$ topology context the converse was proved by R. Ma\~n\'e in \cite{Ma75}. Since hyperbolicity implies a $``C^1 "$ regularity, an open question regarding this approach is to wonder about lower  regularity  than Lipschitz. In this context nothing is known.

In the present article we deal with that type of problem using H\"older regularity instead of Lipschitz, meaning that it is required that  $d(\varphi, Id) \leq  d(f,g)^\alpha$ for some $\alpha<1$. We will restrict ourselves  to a particular class of diffeomorphisms that preserves a one-dimensional subbundle and we will consider only certain type of perturbations instead of any possible perturbation in a neighborhood of an initial system. In this sense, we only assume stability for that type of perturbations. Latter, we apply that framework to the class of partially hyperbolic diffeomorphisms.

Let $\la$ be a compact invariant transitive set such that it has an $Df$-invariant one dimensional sub-bundle $E$. We define the following concepts:

\begin{enumerate}
\item We say that $\la$ is {\em $E$-orientable} if there exist $X^E$ a continuous vector field defined on $\la$ such that $X^E(x)\in E(x)$ and $|X^E(x)|= 1$ for all $x\in \la$.  
\item We say that $f$  {\em preserves the $E$-orientation} of $\la$ if there exist $a:\la \to \R$ a continuous map such that $Df_x(X^E(x))= a(x)X^E(f(x))$ and $a$ is a positive map.
\item An  {\em $E$-translation} is a $C^r$ flow $\psi_s:M\to M,$ $s\in (-\epsilon,\epsilon)$, associated to a $C^r$ vector field $X$ such that the map $<X^E,X>:\la\to \R$ is positive.
\item A perturbation by a  {\em $C^r$ $E$-translation} is a one-parameter family of $C^r$ diffeomorphisms $f_s$ such that $f_s = \psi_s \circ f$ where $\psi_s$ is a $C^r$ $E$-translation.
\item We say that $\la$ is  {\em stable by $C^r$ $E$-translations} if there exist a perturbation by a $C^r$ $E$-translation such that for every $s\in (-\epsilon,\epsilon)$ there exist $\varphi_s:\la\to M$ which verifies:
\[ \varphi_s \circ f = f_s \circ \varphi_s. \] Observe that it is not required that $\varphi_s$ is unique.
\item We say that $\la$ is  {\em $\alpha$-stable by $C^r$ $E$-translations} if there exist a perturbation by a $C^r$ $E$-translation for which is stable 
and also there exist $C>0$ such that
\[d(i,\varphi_s) \leq C|s|^\alpha, \]
where $i:\la \to M$ is the inclusion. Observe that is not required any regularity in the sub-bundle $E$ neither is required that the perturbation preserves such sub-bundle or a continuation of it.
\item We say that $\la$ is  {\em $E$-hyperbolic} if one of the following happens:
\[ lim_{n\in \N}\left\|Df^n_{|E(x)}\right\|=0 \quad \forall x\in \la,\]
or
\[ lim_{n\in \N}\left\|Df^{-n}_{|E(x)}\right\|=0 \quad \forall x\in \la.\]
\end{enumerate}

We are now in condition to enunciate the main theorem of this article.

\begin{teoprin} \label{teoPrin}
Let $f$ be a $C^r$ diffeomorphism with $r\geq 2$ and $\la$ a compact invariant and transitive set with a one dimensional $Df$-invariant sub-bundle $E$. If $\la$ is $E$-orientable, with orientation preserved by $f$ and $\la$ is $\alpha$-stable by $C^r$ $E$-translations with $\alpha> 1/2$ then $\la$ is $E$-hyperbolic.
\end{teoprin}

Coming back to the stability problem, it is  said that a compact invariant  set $\la$ of $f$ is {\em $\alpha$-H\"older-stable by $C^r$ perturbations} if there exist $\U(f)$ a neighborhood of $f$ in the $C^r$-topology such that for every $g\in \U(f)$ we have an homeomorphism $\varphi:\la\to M$ which verifies $\varphi \circ f_{|\la} = g \circ \varphi$ and $d(\varphi,i)\leq Cd(f,g)^\alpha$ for certain $C>0$ where $i:\la\to M$ is the inclusion. Since looking to a one-parameter family of perturbations is weaker than looking to a neighborhood of the system we conclude:

\begin{coro} \label{coroPrin2}
Let $f$ be a $C^r$ diffeomorphism with $r\geq 2$ and $\la$ a compact invariant and transitive set with a one dimensional $Df$-invariant sub-bundle $E$. If $\la$ is $E^c$-orientable, $f$ preserves the $E^c$-orientation of $\la$ and $\la$ is  $\alpha$-H\"older-stable by $C^r$ perturbations with  $\alpha> 1/2$ then $\la$ is hyperbolic.
\end{coro}

We revisit the theorem for the context of partially hyperbolic sets. Recall that an invariant set $\la$ of a diffeomorphisms $f$ is called partially hyperbolic if there is a tangent bundle decomposition $T_\La=E^s\oplus E^c\oplus E^u$ over $\La$ such that vectors in $E^s (E^u)$ are forward (backward) contracted by $Df$ and for any $x$ follows that $||D_xf_{/E^s}||< ||D_{f(x)}f^{-1}_{/E^c}||^{-1}$ and $||D_xf_{/E^c}||< ||D_{f(x)}f^{-1}_{/E^u}||^{-1}.$  From this result in the partially hyperbolic sets context we can conclude the following:

\begin{coro} \label{coroPrin}
Let $f$ be a $C^r$ diffeomorphism with $r\geq 2$ and $\la$ a partially hyperbolic set with  $dim(E^c)=1$. If $\la$ is $E^c$-orientable with orientation preserved by  $f$ and $\la$ is $\alpha$-stable by $C^r$ $E^c$-translations with $\alpha> 1/2$ then $\la$ is hyperbolic.
\end{coro}

In the partially hyperbolic context,  one could consider $E^c-$translation from a flow such that its trajectories move  along the central leaves. Since the foliation is not necessarily $C^r$, that perturbation is not $C^r$. Nonetheless, with the definitions we gave  it is not necessary to consider such type of restricted perturbations.

Coming back to the stability problem and as in corollary \ref{coroPrin2} we conclude:

\begin{coro} \label{coroPrin2ph}
Let $f$ be a $C^r$ diffeomorphism with $r\geq 2$ and $\la$ a partially hyperbolic set with $dim(E^c)=1$. If $\la$ is $E^c$-orientable, $f$ preserves the $E^c$-orientation of $\la$ and $\la$ is $\alpha$-H\"older-stable by $C^r$ perturbations with $\alpha> 1/2$ then $\la$ is hyperbolic.
\end{coro}

Previous corollary can be extended to a more general context: partially hyperbolic sets with many one dimensional center directions:

\begin{coro} \label{coroPrin3}
Let $f$ be a $C^r$ diffeomorphism with $r\geq 2$ and $\la$ a partially hyperbolic set such that $E^c= E^c_1\oplus \dots E^c_k$ where $dim(E^c_i)=1$. If for all $i$, $\la$ is $E_i^c$-orientable, $f$ preserves the $E^c_i$-orientation of $\la$ and  is $\alpha$-stable by $C^r$ $E^c_i$-translations with $\alpha> 1/2$ then $\la$ is hyperbolic.
\end{coro}

Those type of hyperbolic sets appears naturally for diffeomorphisms $C^1$ far from tangencies: in \cite{CSY11} S. Crovisier, M. Sambarino and D. Yang proved that far from tangencies in $Diff^1(M)$ there exist a filtration $\phi=U_0\subset  U_1\subset \dots\subset U_k=M$ such the maximal invariant set in each $U_i\backslash U_{i-1}$ is a partially hyperbolic set where $E^c$ can be splitted in one-dimensional invariant sub-bundles.

The following examples shows that previous theorem are sharp: Let  $f:\R \to \R$ be defined by $f(t)= t - t^2$, and let us restrict the dynamic to a neighborhood of $0$ where it is injective; observe that  $0$ is a fixed point with $f'(0) = 1$ and therefore is not hyperbolic. If we perturb $f$ taking $f_s= f+s$ observe that  the continuation $0(s)$ of $0$ as a fixed point   verifies $0(s)= 0(s) - 0(s)^2 + s$ obtaining $0(s) = \pm s^{1/2}$ when $s\geq 0$. In this example there is no continuation for $s<0$ and there are two continuations for $s>0$. However, the hypothesis needed for the proof of  theorem \ref{teoPrin} are weaker: it is only needed   the existence of a continuation (which might not be only one, recall item 5 in the list of definitions) for $s\geq 0$. 

In particular, this example suggest that the same problems can be approached now requiring and extra hypothesis:  the existence of a unique continuation for on a open set of parameters around $0$. In that case, observe that  the map $f:\R\to \R$ defined by $f(t)= t - t^3$ verifies that it has $0$  as a non hyperbolic fixed point and taking $f_s= f+s$ observe that  the continuation $0(s)$ of $0$ has  a unique continuation that verifying that  $0(s)=s^{1/3}$. This open the question if the extra hypothesis could allow us to improve the result for $\alpha > 1/3$. Even though this example is not Kupka-Smale, one which actually is can be done in dimension $3$ using the ideas of the example in \cite{Pu06}. This example is essentially a skew-product which has the previous map over a minimal set of a two dimensional horseshoe. All the previous examples have the property that even though they are not hyperbolic, the perturbation by the central translation force them to be hyperbolic. We wonder if this case, Even if it is not, we can formulate a weaker question: Given a partially hyperbolic set $\la$ stable by central translation, are the parameters in which the continuation of $\la$ is hyperbolic an open and dense set? 

A particular class of diffeomorphisms having a one dimensional invariant sub-bundle  are the skew-product. Let $h:M\to M$ be a $C^r$ diffeomorphism, and $g:M\times R\to \R$ be a $C^r$ map such that the maps $g_x:R\to \R$, defined by $g_x(t)=g(x,t)$, are $C^r$ diffeomorphisms $\forall x\in M$, we define the skew-product family as the maps $H:M\times \R\to M\times \R$ with $H(x,t)=(h(x),g(x,t))$. In particular we define for them  the sub-bundle $E$ given by  $\{0\}\times T_t\R\subset T_{(x,t)}M\times \R$.  All previous results can be applied for that class where the natural translation family is given by $H_s(x, t)=(h(x),g(x,t)+s).$

For that particular class with center dynamics  given by a family of one dimensional diffeomorphisms, we wonder if for certain class of dynamics   previous theorem can be extended for $\alpha\leq 1/2$ or even for the case that it is  only assume continuous variation of the conjugacy respect  the parameter perturbation (in such a case we say that $H$ is {\em continuously stable}).

Given  a $C^2$ skew-product $H=(h,g)$ it is defined the maps $g',g'':M\times \R \to \R$ where $g'(x,t)= \frac{\partial g}{\partial t}(x,t)$ and $g''(x,t)=\frac{\partial^2 g}{\partial^2 t}(x,t)$. It is clear that if $H$ preservers the $E$-orientation on $\la$ a compact invariant set then $g'_{|\la}>0$. 

We aim now to study the case of continuous stability by the central translation. Since the proof for H\"older stability only holds with $\alpha>1/2$, we aimed to work with extra hypothesis on the system instead of the stability of the set.

\begin{teoprin} \label{segpos}Let $H(x,t)$ be a $C^2$ Kupka-Smale skew-product  and $\la$ a compact invariant and transitive set such that $H$ preserves the $E$-orientation and $g''_{|\la}\geq 0$ on a neighborhood of $\la$. If $\la$ is stable by $C^r$ $E$-translation and the periodic points are dense then it is $E$-hyperbolic.
\end{teoprin}

 A simple hypothesis that provides good distortion estimates and  allows that the second derivative can change from positive to negative in different fibers,  is to assume that the one dimensional dynamics fibers has negative Schwartzian derivative, i.e. for any $g_x$ it is supposed that $Sg_x(t)= \frac{g_x^{'''}(t)}{g_x^{''}(t)}-\frac{3}{2}\left(\frac{g_x^{''}(t)}{g_x^{'}(t)}\right)^2<0$ for any $t$ in a neighborhood of $\la$. A simple result assuming that hypothesis is the following:

\begin{teoprin}\label{sch neg} Let $H(x,t)$ be a  $C^3$ skew-product such that for any $x$ the diffeomorphism $g_x$ verifies that $Sg_x<0.$ Let $\la$ be a transitive set and if $H_{|\la}$ is continuously stable then $\la$ is $E$-hyperbolic.
\end{teoprin}

The negative Schwartzian derivative hypothesis is commonly used in one-dimensional dynamics. Singer in \cite{Si78} proved that for one dimensional maps  with negative Schwartzian derivative if all the critical points belong to the basin of some hyperbolic attractor then the system is hyperbolic. Also, in the context of skew-products, Solano in \cite{So12} proved that non-uniform hyperbolicity along the leaves implies existence of a finite number of ergodic absolutely continuous invariant probability measures which describe the asymptotic of almost every point. 

Regarding the orientability concepts, it is clear that for the skew-product family any compact invariant set $\la$ will be $E$-orientable. Also if $\la$ is a compact connected  $f$-invariant set and $E$ is a one dimensional sub-bundle for which $\la$ is $E$-orientable then $f$ will always preserve or reverse the orientation of $E$. 

We address in section 2 with the H\"older Stability result and in section 3 with the continuously stable results.

\textbf{Acknowledgment:} Both authors were supported by CNPq, Brazil.

%%%%%%%%%%%%%%%%%%%%%%%%%%%%%%%%%%%%%%%%%%%%%%%%%%%%%%%%%%%%%%%%%%%PROOF OF THE THEOREM%%%%%%%%%%%%%%%%%%%%%%%%%%%%%%%%%%%%%%%%%%%%%%%%%%

\section{H\"older Stability - Proof of Theorem 1}

To prove this theorem we use the notion of bounded solution of a family of linear isomorphisms and we adapt the techniques used by S. Tikhomirov in \cite{Ti11}. In that article it is considered systems having H\"older shadowing property for finite intervals. Let us remark that orbit of a perturbation is always pseudo-orbit of the original system yet not all pseudo-orbits are associated to a $C^1$ perturbation and much less to a $C^r$ perturbation. In relation with what we said in the introduction, hyperbolic systems have a slow movement in the continuation of the points. The algebraic environment we are going to use allows us to control such movement and to prove the converse: slow movement implies hyperbolicity. Differently to what is done in  \cite{Ti11} where the problem is about  shadowing property and  algebraic perturbations can be projected on the ambient manifold  to become become  pseudo-orbits, we need to consider algebraic perturbations that induces translation in the ambient manifold (recall that we only consider perturbation obtained by translations). To deal with that, we prove that the worst perturbation of this algebraic environment is the uniform algebraic translation and we also prove that the nearby perturbations of the uniform translation are  comparable to it. We prove then that the $E$-translations are associated to algebraic translations nearby the uniform translation and therefore we conclude hyperbolicity in the sub-bundle $E$.

The general framework is the following: Let $\{E_m\}_{m\in \Z}$ be a family of euclidean spaces of dimension $k$ and $\A=\{A_m:E_m\to E_{m+1}\}_{m\in \Z}$ a sequence of linear isomorphism such there exist $R>0$ with $\left\|A_m\right\|< R$ and $\left\|A_m^{-1}\right\|< R$ for all $m\in \Z$. We say that $\A$ has bounded solution if there exist $Q>0$ such that for every $m\in \Z$ and $n\in \N$ if we take $w_{m+1},\dots,w_{m+n}$ with $w_i\in E_i$ and $\left|w_i\right|\leq 1$ there exist $v_{m},\dots,v_{m+n}$ with $v_i\in E_i$ which verifies:
\[ v_{i+1}= A_i v_i + w_{i+1}\quad \forall i=m,\dots, m+n-1,\]
and 
\[|v_i| \leq Q \quad \forall i=m,\dots, m+n.\]

What the previous definition says is that if we perturb the linear isomorphisms by translations in finite intervals we can find a continuation of the $0$ orbit at bounded distance. 
What is proven by D. Todorov in $\cite{To13}$ is that bounded solution implies hyperbolicity for $\A$ which means that $E_m = E^s_m\oplus E^u_m$, $A_m(E^{s,u}_m)= E^{s,u}_{m+1}$ and there is exponential contraction on $E^s$ and exponential expansion on $E^u$. Since in our context $E_m$ has dimension 1, proving hyperbolicity from bounded solution is simple and so we do it at the end of the article.   

Given $x\in \la$ we define $\A(x)=\{Df_{|E}:E(x_m)\to E(x_{m+1})\}$ where $x_m=f^m(x)$. We say that $\la$ has uniform bounded solution if there exist $Q$ such that $\A(x)$ has bounded solution and $Q$ is a bound for all $ x\in \la $. 

\begin{prop}\label{propPrin}
Let $f$ be a $C^r$ diffeomorphism with $r\geq 2$ and $\la$ a compact invariant set with $E$ a one dimensional $Df$-invariant sub-bundle. If $\la$ is $E$-orientable, $f$ preserves or reverses the $E$-orientation of $\la$ and $\la$ is $\alpha$-stable by $C^r$ $E$-translations with $\alpha> 1/2$ then $\la$ has uniform bounded solution.
\end{prop}

\begin{proof}
The uniformity will come along the proof. It is just needed to see that the constants do not depend on $x$. We will fix $x$ and prove that $\A(x)=\A$ has bounded solution.

Observe first that given $v\in E(x_m)$, we have that $v= <v,X^E(x_m)> X^E(x_m)$ and that $Df(v)= a(x_m)<v,X^E(x_m)>X^E(x_{m+1})$ where $a(x_m)$ comes from the definition of $f$ preserving or reverting the $E$-orientation. To simplify the notation we will identify $E(x_m)=E_m$ with $\R$ and $A_n$ with the linear map $v\mapsto a(x_m)v= a_m v$.

Given $m\in\Z$ and $n\in \N$, take $w_{m+1},\dots,w_{m+n}\in \R$ $(\{w_i\}_{m,n})$ and also $v_m,\dots, v_{m+n}\in \R$ $(\{v_i\}_{m,n})$ such that
\[v_{i+1}= a_i v_i + w_{i+1}\quad \forall i=m,\dots, m+n-1. \] 

We define the norm $\left\| \{v_i\}_{m,n}\right\|= max\{|v_i|:i=m,\dots, m+n\}$.

To prove the proposition we need to find a number $Q>0$ which for all $m\in \Z$, $n\in \N$ and all $ \{w_i\}_{m,n}$ with $\left\|\{w_i\}_{m,n}\right\|\leq 1$ we can find $\{v_i\}_{m,n}$ such that $\left\|\{v_i\}_{m,n}\right\|\leq Q$. 

It will be clear in the proof that the starting point of the sequences $\{w_i\}_{m,n}$ and $\{v_i\}_{m,n}$ will not be relevant in the computations, therefore we assume $m=0$ and from now on we note $\{w_i\}_n$ and $\{v_i\}_n$.

In order to find $Q$, given $\{w_i\}_n$ we define 
\[O(\{w_i\}_n)=\{\{v_i\}_n: v_{i+1}= a_i v_i + w_{i+1}\ \forall i=m,\dots, m+n-1\},\]
as the space of finite orbits for the perturbation $\{w_i\}_n$. Since we want to find one $\{v_i\}_n\in O(\{w_i\}_n)$ with a small norm we will take the one with the smallest. We define then
\begin{equation}
P(\{w_i\}_n)=min\{\left\|\{v_i\}_n\right\|: \{v_i\}_n\in O(\{w_i\}_n)\}.
\label{defP}
\end{equation}

Since $\left\|\cdot\right\|$ is a norm the previous definition is good. Now we take the worst perturbation and define
\[ Q(n)=max\{P(\{w_i\}_n):\left\|\{w_i\}_n\right\|\leq 1\}.\]

The previous definition is good because $P$ is continuous according to $\{w_i\}_n$ and the space of $\{w_i\}_n$ with $\left\|\{w_i\}_n\right\|\leq 1$ is compact.

The first property (P1) is the following, given $d\in \R$ and $\{w_i\}_n$ we have that $P(\{dw_i\}_n) = |d|P(\{w_i\}_n)$. This is because $\{v_i\}_n\in O(\{w_i\}_n)\Longleftrightarrow \{dv_i\}_n\in O(\{dw_i\}_n)$ and $\left\|\{dv_i\}_n\right\|=|d|\left\|\{v_i\}_n\right\|$.

From this we conclude the property (P2): Given $\{w_i\}_n$ with $\left\|\{w_i\}_n\right\| > 0$ there exist $\{v_i\}_n\in O(\{w_i\}_n)$ such that 
\[ \left\|\{v_i\}_n\right\|\leq Q(n) \left\|\{w_i\}_n\right\|. \]

To see this take $d= \left\|\{w_i\}_n\right\|$ and then $\left\|\{d^{-1}w_i\}_n\right\|\leq 1$. By definition of $Q(n)$ we have that 
$P (\{d^{-1}w_i\}_n)\leq Q(n)$ and therefore there exist $\{\hat{v}_i\}_n\in O(\{d^{-1}w_i\}_n)$ such that $\left\|\{\hat{v}_i\}_n\right\|\leq Q(n)$. If $\{v_i \}_n=  \{d\hat{v}_i\}_n$ then $\{v_i \}_n\in O(\{w_i\}_n)$ and from the property (P1) we conclude that $\left\|\{v_i\}_n\right\|\leq Q(n) \left\|\{w_i\}_n\right\| $.

We now prove the main lemma: {\em  in the present  algebraic context, the algebraic uniform translation ($\{1\}_n$) is the worst perturbation}. Later we show that  the algebraic uniform translation  can be associated to $C^r$ perturbation.  Using that we are dealing with one-dimensional perturbations and $Q(n)$ measures how far we can find the continuation of an orbit, it follows that the most far continuation is obtained translating to the same side:  if not, the continuation should be closer.

\begin{lema}(Main-Lemma)
If $a_i>0$ $\forall i$ then $Q(n)= P(\{1\}_n)$ where $P$ is given by (\ref{defP}).
\end{lema}

\begin{proof}
Given $\{w_i\}_n$ with $\left\|\{w_i\}_n\right\|\leq 1$ if $\{v_i\}_n \in O(\{w_i\}_n)$ we have that 
\[ v_i =  \prod_{k=0}^{i-1}a_k v_0+ \sum_{k=1}^{i}\prod_{j=k}^{i-1}a_j w_k. \]

If $B_i = \prod_{k=0}^{i-1}a_k$ and $C_i = \sum_{k=1}^{i}\prod_{j=k}^{i-1}a_j w_k$ we define $g_i:\R\to \R$ and $G_n:\R\to \R$ such that
\[g_0(v)=v\]
\[g_i(v)= B_i v+ C_i\quad \forall i=1,\dots,n,\]
and 
\[G_n(v)=max\{|g_i(v)|:0\leq i\leq n\}.\]

If $D_i = \sum_{k=1}^{i}\prod_{j=k}^{i-1}a_j$ in an analogous way we define $f_i:\R\to \R$ and $F_n:\R\to \R$ such that
\[f_0(v)=v,\]
\[f_i(v)= B_i v+ D_i\quad \forall i=1,\dots,n,\]
and 
\[F_n(v)=max\{|f_i(v)|:0\leq i\leq n\}.\]

Therefore we have that $P(\{w_i\}_n)= min\{G_n(v):v\in \R\}$ and $P(\{ 1\}_n)= min\{F_n(v):v\in \R\}$.

Given $i_1,i_2\leq n$ we define the maps $(g_{i_1},g_{i_2}):\R \to \R$ and $(f_{i_1},f_{i_2}):\R \to \R$ by 
\[(g_{i_1},g_{i_2})(v)= max\{|g_{i_1}(v)|,|g_{i_2}(v)|\},\]
and 
\[(f_{i_1},f_{i_2})(v)= max\{|f_{i_1}(v)|,|f_{i_2}(v)|\}.\]

We also define the values $min(g_{i_1},g_{i_2})$ and $min(f_{i_1},f_{i_2})$ as the minimum value taken by the maps $(g_{i_1},g_{i_2})$ and $(f_{i_1},f_{i_2})$ respectively. It is easy to verify that the infimum value is in fact a minimum.

Since $|g_i|$ and $|f_i|$ are convex functions $G_n$ and $F_n$ are also convex functions. This implies the following assertion:
\[min(G_n)= max\{ min(g_{i_1},g_{i_2}):i_1, i_2 \leq n\},\]
and 
\[min(F_n)= max\{ min(f_{i_1},f_{i_2}):i_1, i_2 \leq n\}.\]

The previous assertion tell us that to compare $P(\{w_i\}_n)$ and $P(\{1\}_n)$ we just need to compare $min(g_{i_1},g_{i_2})$ and $min(f_{i_1},f_{i_2})$.

We are going to prove now that for $i_1$ and $i_2$ fixed we have that 
\[ min(g_{i_1},g_{i_2})\leq min(f_{i_1},f_{i_2}).\]

This and the previous assertion implies $P(\{w_i\}_n)\leq P(\{1\}_n)$ which concludes the lemma.

For the maps $g_i$ and $f_i$ we have the following property: Given $k,l\in \N$ such that $k+l\leq n$ there exist $B_{k,l},C_{k,l}$ and $D_{k,l}$ such that:
\[g_{k+l}(v)= B_{k,l}g_k(v) + C_{k,l},\]
and
\[f_{k+l}(v)= B_{k,l}g_k(v) + D_{k,l}.\]

Let us observe now that $D_i$ is always positive. In particular it verifies $D_i \geq |C_i|$ and moreover $D_{k,l}\geq |C_{k,l}|$.
The previous statements are easy computations concluded from the fact that $a_i>0$.

Fix now $i_1$ and $i_2$. Suppose that $i_1<i_2$ and take $k=i_1$ and $l=i_2-i_1$. We have then:
\[g_{i_1}(v)= B_kv + C_k\quad g_{i_2}(v)= B_{k,l}B_kv + B_{k,l}C_k + C_{k,l},\]
and
\[f_{i_1}(v)= B_kv + D_k\quad f_{i_2}(v)= B_{k,l}B_kv + B_{k,l}D_k + D_{k,l}.\]

If $min(g_{i_1},g_{i_2})= (g_{i_1},g_{i_2})(v_0)$ then $v_0$ verifies $|g_{i_1}(v_0)| =|g_{i_2}(v_0)|$. Moreover if 
$\hat v_1$ and $\hat v_2$ are such that $g_{i_j}(\hat v_j)=0$ then $v_0\in [min\{\hat v_1,\hat v_2\},max\{\hat v_1,\hat v_2\}]$.
Suppose that $\hat v_1<\hat v_2$ then $v_0$ verifies the equation:
\[ g_{i_1}(v_0)=-g_{i_2}(v_0).\]

If we resolve this we conclude that
\[v_0= \frac{-C_k -C_{k,l} - B_{k,l}C_k}{B_k+B_{k,l}},\]
and therefore
\[min(g_{i_1},g_{i_2})= \frac{-C_{k,l}B_k}{B_k+B_{k,l}}.\]

Since $B_k$ and $B_{k,l}$ are positive $C_{k,l}$ must be negative. This comes from the condition $\hat v_1<\hat v_2$. In any case we have that
\[min(g_{i_1},g_{i_2})= \frac{|C_{k,l}|B_k}{B_k+B_{k,l}}.\]

Computing for $f_{i_1}$ and $f_{i_2}$ we conclude
\[min(f_{i_1},f_{i_2})= \frac{D_{k,l}B_k}{B_k+B_{k,l}}.\]

Since $D_{k,l}\geq |C_{k,l}|$ we have that $min(f_{i_1},f_{i_2})\geq min(g_{i_1},g_{i_2})$ finishing the proof of the lemma
\end{proof}

We now prove that we can compare algebraically the dynamics of the perturbation $\{1\}_n$ with a close one.

\begin{lema} \label{lemaPrin}
If $\{w_i\}_n$ verifies $w_i \geq d$ for some $d>0$ then $d Q(n)\leq P(\{w_i\}_n)$.
\end{lema}

\begin{proof}
Define $\hat w_i = w_i d^{-1}$. Then $\hat w_i\geq 1$ $\forall i=1,\dots,n$. Take as in the previous lemma $g_i$, $G_n$, $B_k$, $C_k$, $C_{k,l}$ and $(g_{i_1},g_{i_2})$ associated to $\{\hat w_i\}_n$ and $f_i$, $F_n$, $D_k$, $D_{k,l}$ and $(f_{i_1},f_{i_2})$ associated again to $\{1\}_n$. In this case we have that $C_k\geq D_k>0$ and moreover $C_{k,l}>D_{k,l}>0$. This implies that $min(g_{i_1},g_{i_2})\geq min(f_{i_1},f_{i_2})$ which implies that $min(G_n)\geq min(F_n)$. Since $min(G_n)= P(\{\hat w_i\}_n)= d^{-1}P(\{ w_i\}_n)$ and $min(F_n)= P(\{1\}_n)=Q(n)$ because of the previous lemma we conclude the proof of the lemma.
\end{proof}

Let us now link this algebraic environment with the dynamics. Recall that $f_s = \psi_s \circ f$, and we have the conjugation $\varphi_s$. We define $x_i(s)=\varphi_s(x_i)$ $\forall i\in \Z$. From now on we will work with $s>0$. By hypothesis we have that $d(x_i,x_i(s))\leq Cs^\alpha$.

Let us suppose that $s$ is small enough such that $Cs^\alpha\leq \delta$ where $\delta$ is a uniform radius for which the map $exp_x:T_xM(\delta)\to B(x,\delta)$ is a  diffeomorphism. 

We define now $u_i(s)= exp^{-1}_{x_i}(x_i(s))$ and we consider its projection to $E(x_i)=E_i$ which we identified with $\R$ by $\hat u_i(s)=<X^E(x_i),u_i(s)>$.

For the following computation we are going to need the next definition: given $i$ and $s$ we define the vector $\hat X_i(s)=D(exp_{x_i}^{-1})_{x_i(s)}(X(x_i(s)))\in T_{x_i}M$.

We have then that
\[\hat u_{i+1}(s)=<X^E(x_{i+1}),u_{i+1}(s)>=<X^E(x_{i+1}), exp^{-1}_{x_{i+1}}(x_{i+1}(s))>\]
\[= <X^E(x_{i+1}), exp^{-1}_{x_{i+1}}(f_s(x_{i}(s)))>=<X^E(x_{i+1}), exp^{-1}_{x_{i+1}}(\psi_s(f(x_{i}(s))))>. \]
Using the Taylor polynomial on $\psi_s$ we estimate $exp^{-1}_{x_{i+1}}(\psi_s(f(x_{i}(s))))$ by $exp^{-1}_{x_{i+1}}(f(x_{i}(s))) + s\hat X_{i+1}(s)$. In particular if $b_{i+1}(s)=  exp^{-1}_{x_{i+1}}(\psi_s(f(x_{i}(s)))) - exp^{-1}_{x_{i+1}}(f(x_{i}(s))) - s\hat X_{i+1}(s)$ then $|b_{i+1}(s)|/s\stackrel{s\to 0}{\longrightarrow} 0$ uniformly on $i$. We proceed with our computation:
\[\hat u_{i+1}(s)=<X^E(x_{i+1}),b_{i+1}(s)+ exp^{-1}_{x_{i+1}}(f(x_{i}(s))) + s\hat X_{i+1}(s)>\]
\[=<X^E(x_{i+1}),b_{i+1}(s)+ exp^{-1}_{x_{i+1}}(f(exp_{x_i}(u_i(s)))) + s\hat X_{i+1}(s)>.\]

Since $f$ is at least $C^2$ we have that $| exp^{-1}_{x_{i+1}}(f(exp_{x_i}(u_i(s)))) - Df_{x_i}(u_i(s))|\leq C_1 |u_i(s)|^2$.
Let $r_{i+1}(s)=exp^{-1}_{x_{i+1}}(f(exp_{x_i}(u_i(s)))) - Df_{x_i}(u_i(s))$ and then we have that
\[\hat u_{i+1}(s)=<X^E(x_{i+1}),Df_{x_i}(u_i(s))+ s\hat X_{i+1}(s)+b_{i+1}(s)+r_{i+1}(s)>.\]
Observe that $<X^E(x_{i+1}),Df_{x_i}(u_i(s))>= a_i \hat u_i(s)$. If we define 
\[\hat w_{i+1}(s) = <X^E(x_{i+1}),s\hat X_{i+1}(s)+b_{i+1}(s)>,\] 
and 
\[\hat r_{i+1}(s)= <X^E(x_{i+1}),r_{i+1}(s)>,\]
then we obtain:
\[ \hat u_{i+1}(s) = a_i \hat u_i(s) + \hat w_{i+1}(s) + \hat r_{i+1}(s).\]

In particular 
\[|\hat r_{i+1}(s)|\leq C_1 |u_i(s)|^2\leq C_1 C^2 s^{2\alpha}.\]

We also have from definition of $X$ and the fact that $|b_{i+1}(s)|/s\stackrel{s\to 0}{\longrightarrow} 0$ uniformly on $i$ that there exist $d>0$ such that
\[\frac{\hat w_{i+1}(s)}{s} > d>0\ \forall s\in (0,\epsilon_0).\]
for certain $\epsilon_0>0$.

If we define $C_2= C_1C^2$ and use the property (P2) from $Q(n)$ to $\{\hat r_i(s)\}_n$ we conclude the existence of $\{e_i(s)\}_n\in O(\{\hat r_i(s)\}_n)$ such that
\[\left\|\{e_i(s)\}_n\right\|\leq Q(n) \left\|\{\hat r_i(s)\}_n\right\|\leq Q(n)C_2 s^{2\alpha}.\]

Is easy to check that $\{\hat u_i(s) - e_i(s)\}_n\in O(\{\hat w_i(s)\})$ and therefore $\{(\hat u_i(s) - e_i(s))/s\}_n\in$ \linebreak
$ O(\{\hat w_i(s)/s\})$.
Since $\hat w_{i+1}(s)/s > d>0$ $\forall i=1,\dots,n$ then $dQ(n)\leq P(\{\hat w_{i+1}(s)/s\})$ as a consequence of lemma \ref {lemaPrin} . But now 
$ P(\{\hat w_{i+1}(s)/s\})\leq \left\|\{(\hat u_i(s) - e_i(s))/s\}_n\right\|$ obtaining that
\[dQ(n)\leq \left\|\left\{\frac{\hat u_i(s) - e_i(s))}{s}\right\}_n\right\|\leq Cs^{\alpha -1} + Q(n)C_2 s^{2\alpha -1},\]
If we take $s$ small enough which we can because $2\alpha -1>0$ then we have that
\[Q(n)\leq \frac{Cs^{\alpha -1}}{d - C_2 s^{2\alpha -1}}.\]
finishing the proof of the proposition \ref{propPrin}.
\end{proof}

Let us prove now the theorem \ref {teoPrin}. 

\begin{proof}
Take $Q$ from the previous proposition which does not depend on $x$. Given $m\in \Z$ and $n\in\N$ we define
\[\lambda(m,n) = \prod_{i=m}^{m+n-1}a_i\]

Let us prove the following lemma:
\begin{lema}
There exist $n_0$ such that for any $m\in \Z$ we have
\[\lambda(m,n_0)> 2 \ or\ \lambda(m+n_0,n_0) < 1/2\]
\end{lema}

\begin{proof}
Given $m\in \Z$ and $n\in \N$, consider the family $\{-1\}_{m,2n}$. Since $\A$ has bounded solution there exist $\{v_i\}_{m,2n}\in O(\{-1\}_{m,2n})$ such that $\left\|\{v_i\}_{m,2n}\right\| \leq Q$.

Observe that since $a_i >0$ if $v_i\leq 0$ then $v_j<0$ for all $j\geq i$. We have two cases now either: $v_{m+n-1} > 0$ or $v_{m+n-1} \leq 0$. For the first one we have that $v_i>0$ for $i=m,\dots,m+n-1$.

From the equation $v_{i+1}=a_iv_i - 1$ we have that $a_i = \frac{v_{i+1} + 1}{v_i}$. Therefore

\[\lambda(m,n)= \prod_{i=m}^{m+n-1}\frac{v_{i+1} + 1}{v_i}= \frac{v_{m+n}}{v_0}\prod_{i=m}^{m+n-1}\frac{v_i + 1}{v_i}= \frac{v_{m+n}}{v_0}\prod_{i=m}^{m+n-1}\left(1 + \frac{1}{v_i}\right)\]

Using the bound $Q$ over $v_i$ we have that 
\[\lambda(m,n) \geq \frac{1}{Q}\left(1+  \frac{1}{Q}\right)^n.\]

If $v_{m+n-1}\leq 0$ then $v_i <0$ $\forall i=m+n,\dots,m+2n-1$ and then 
\[\lambda(m+n,n)\leq (Q+1)\left(1 - \frac{1}{Q}\right)^n.\]
Taking $n_0$ big enough we conclude the proof of the lemma.
\end{proof}

From the previous lemma is easy to see that if $\lambda(m,n_0)>2$ then $\lambda(m-kn_0,n_0)>2$ for all $k\in \N$ and if $\lambda(m,n_0)<1/2$ then $\lambda(m+kn_0,n_0)<1/2$. Define now $\la^u=\{x\in \la:\lambda(x,m,n_0)>2\}$ and $\la^s=\{x\in \la:\lambda(x,m,n_0)<1/2\}$. Due to continuity and the previous assertion we conclude that this two sets are compact, invariant and disjoint and therefore one must be empty. Having that $\la = \la^s$ or $\la =\la^u$ implies that $\la$ is $E$-hyperbolic.
\end{proof}

\section{Continuous Stability - Proof of Theorems 2 and 3}

The proofs of both theorems share a core understanding of the dynamics of a stable set by $C^r$ $E$-translations. In particular the orientation preserving hypothesis has a key role to establish a relationship between the direction of the continuation of the set and whether the set is expanding or contracting along the $E$ sub-bundle.

A $C^r$ $E$-translation for a skew-product $H = (h,g)$ has the form 
\[H_s(x,t)=(h(x,t,s),g(x,t) + g_1(x,t,s)),\] 
where 
\[\frac{\partial g_1}{\partial s}(x,t,0) > 0\ and\ h(x,t,0)=h(x).\]  
 
Since the action of the perturbation transversally to $E$ is irrelevant to us, we may assume that $h(x,t,s)=h(x)\ \forall s\in (-\epsilon,\epsilon)$. Taking the Taylor polynomial of first degree $g_1$ can be seen as $g_1(x,t,s) = s c_1(x,t) + r(x,t,s)$ where $\delta_1> c_1(x,t) > \delta_0> 0\ \forall (x,t)$ in a neighborhood of $\la$ for some $\delta_1> \delta_0 > 0$. To simplify the computations we may assume that $g_1(x,t,s) = s$ yet the proofs from now on also holds for any $C^r$ $E$-translation.

Take $H=(h,g)$, $\la\subset M\times \R$ a compact invariant and transitive stable set by $C^r$ $E$-translations with $\varphi_s$ the conjugacy. We will use the notation $z(s)= \varphi_s(z)$ and $\nu(z,s)=\pi_\R(z(s))-\pi_\R(z)$.

\begin{lema}($\nu$-lemma)\label{nulema}
Let $H(x,t)$ be a $C^1$ skew-product and $\la$ a compact invariant and transitive set such that $H$ preserves the $E$-orientation. If $\la$ is continuous stable by $C^r$ $E$-translation then either $\nu(z,s)\geq s\ \forall z\in \la,\ s\in (0, \epsilon)\  or\ \nu(z,s)\leq -s\ \forall z\in \la,\ s\in (0, \epsilon)$.
\end{lema}

\begin{proof}
Observe first that if $z\in \la$ and we have that $z(s)$ verifies $\pi_\R(z(s))>\pi_\R(z)$ then $\pi_\R(H_s(z(s))) = g(z(s))+ s > g(z) + s = \pi_\R(H(z)) + s$ due to the fact that $g'(z) > 0$. This implies that $\nu( H(z),s) \geq s$.  In an analogous fashion if $\pi_\R(z(s))<\pi_\R(z)$ then $\nu( H^{-1}(z),s) \leq -s$.

Since $\varphi_s$ is continuous $\nu$ is also continuous. From the previous assertion if we take a point $z_0$ with dense orbit on $\la$ and there exist $n\in \Z$ with $\nu(H^n(z_0),s)>0$ then we conclude that $\nu(z,s)\geq s\ \forall z \in \la$. If  $\nu(H^n(z_0),s)<0$ then $\nu(z,s)\geq s\ \forall z \in \la$. Observe also that if $\nu(z,s)=0$ then  $\nu(H(z),s)>0$. From this we rule out the possibility of having $\nu(z,s)=0\ \forall z\in \la$ concluding the lemma.
\end{proof}

Let us focus our attention into proving theorem \ref{segpos}. Suppose now that $\la$ has dense periodic points and $H$ is also Kupka-Smale. 

It is a known fact that the continuation of a hyperbolic periodic points is as differentiable as $H$ due to the implicit function theorem. Given a periodic point $p$, if $p(s)$ is the continuation of $p$, we note $p'(s)$ the projection to $\R$ of the first derivative. From the equation $H_s(p(s))= H(p)(s)$ and taking the first derivative on $s$ we obtain the equation:
\[H(p)'(s)=g'_s(p(s)) p'(s) + 1.\]

Now using an inductive argument we can conclude that:
\begin{equation}\label{velPP}
p'(s)= \frac{\sum_{i=0}^{n-1} \prod_{j=i+1}^{n-1}g_s'(H^i(p)(s))}{1 - \prod_{i=0}^{n-1}g_s'(H^i(p)(s))}.
\end{equation}

Although we are not going to use directly the following result it is relevant to the understanding of the dynamics of a set which is stable by $E$-translations. Using the previous equation and the $\nu$-lemma we prove that: 

\begin{prop}\label{prop A}
Let $H(x,t)$ be a $C^1$ Kupka-Smale skew-product and $\la$ a compact invariant and transitive set such that $H$ preserves the $E$-orientation. If $\la$ is stable by $C^r$ $E$-translation and the periodic points are dense then all of the periodic points are attracting or all are repelling in the $E$-direction. Moreover the continuation of the periodic points remains being attracting or repelling depending on the initial circumstance. Also $\nu(z,s)>0 \Leftrightarrow$ all the periodic points are contractive in the $E$-direction.   
\end{prop}

\begin{proof}
 Due to the hypothesis $g'_s>0$ in the equation \ref{velPP}, we have  $\sum_{i=0}^{n-1} \prod_{j=i+1}^{n-1}g_s'(H^i(p)(s))>0$ and also $\prod_{i=0}^{n-1}g_s'(H^i(p)(s)) > 0$. So $p'(s)< 0$ if and only if $\prod_{i=0}^{n-1}g_s'(H^i(p)(s)) > 1$. 

Now taking the Taylor polynomial of first degree on $p(s)$ and using the lemma \ref{nulema} we conclude that $p'(0)$ have all the same sign. If we have that $p'(s)$ changes sign at some $s_0$ (in $s_0$, $p(s)$ is not differentiable) then applying  lemma \ref{nulema}  to $H_{s_0}$ we conclude that all the periodic points change sign at $s_0$. In that case we can just restrict our study to $s\in (-s_0,s_0)$. Clearly those turning points can not accumulate over $0$ because we started at a Kupka-Smale skew-product. 
\end{proof}

Using the previous proposition and studying the evolution of the Lyapunov exponents we could prove a weaker version of Theorem $2$, asking $g''>0$ instead of $g''\geq 0 $ in a neighborhood of $\la$. A scheme to prove that is the following: The compactness of $\la$ implies that $g''>\delta >0$ for some $\delta>0$. This will imply that the Lyapunov exponents of the periodic points grow at a uniform speed. By this we mean that if 
\[ \lambda(p(s),H_s)= \frac{log \left(\prod_{i=0}^{per(p)-1} g_s'(H_s^i(p(s))\right)}{per(p)}, \]
then $\lambda(p(s),H_s) \geq  \lambda(p,H) + Cs$ for some $C>0$ if $\nu > 0$ or  $\lambda(p(s),H_s) \leq  \lambda(p,H) - Cs$ if $\nu < 0$. Now if all the periodic points are contracting they remain contracting and $\nu > 0$, therefore $\lambda(p,H) + Cs<0$. From this we conclude that the periodic points are a $E$-hyperbolic and this can be extended to its closure $\la$. The uniform growth can be obtained by computing $g'_s(H^i_s(p(s))$ through the Taylor polynomial of first degree of $g'$ at 
$H^i(p)$.

Using only the study of the continuation of the periodic points and the $\nu$-lemma we will prove Theorem \ref{segpos}.

\begin{proof}
Suppose that $\nu >0$. If $\la$ is not hyperbolic then there exist a point $z_0\in \la$ such that $lim inf \prod_{i=0}^{n-1}g'(H^i(z_0))\neq 0$. Since $g'>0$ we can assume there exist $\delta>0$ such that $\prod_{i=0}^{n-1}g'(H^i(z_0))>\delta >0$ for all $n>0$.

Since we are assuming that all the periodic points are contractive in the $E$-bundle, we can reformulate the equation \ref{velPP} of the speed of a periodic point obtaining that: 
\[p'(0)= \sum_{i=0}^{\infty} \prod_{j=1}^{i-1}g'(H^{-j}(p))\]

Using $z_0$, the density of the periodic points and the orientation preserving hypothesis we can construct a family of periodic points $p_n$ such that $\forall n >0$ there exist $K_n>0$ with  
\[\prod_{j=1}^{i-1}g'(H^{-j}(p_n)) > \delta/2\ \forall\ 0\leq i\leq K_n\ \forall n>0,\]
with $lim_n\ K_n = \infty$. 

This in particular implies that $p_n'(0) \geq K_n \delta/2$ and therefore $lim_n\ p_n'(0) = \infty$. We now use the hypothesis $g''\geq 0$ in a neighborhood of $\la$ to see that $g'(p(s))$ is a non-decreasing function for all periodic point. This two things implies that $lim_n p_n'(s)=\infty$ for $s>0$. We can take if necessary a sub sequence of such $p_n$ in order to converge to a point $z_1$. The stability hypothesis implies that $p_n(s)$ converges to $z_1(s)$, but the maps $p_n(s)$ can not converge point-wise to any map because $lim_n\ p_n'(s)=\infty$ obtaining a contradiction. 

The case $\nu <0$ is similar, the important remark is that we need to use another reformulation of the equation \ref{velPP} associated to expanding periodic points in the $E$-direction which is: 
\[p'(0)= \sum_{i=1}^{\infty} \prod_{j=0}^{i-1}g'(H^{j}(p))^{-1},\]
and the rest of the proof goes straightforward.
\end{proof}

Let us focus now into proving Theorem \ref{sch neg}. 

Suppose now that $H(x, t)= (h(x), g(x,t))$ is a $C^3$ skew-product. For every $x\in M$ we define the map $Sg_x:\R\to \R$ by  $\frac{g_x'''}{g_x''}-\frac{3}{2}
\left(\frac{g_x''}{g_x'}\right)^2.$ Suppose that $\la\subset M\times\R$ is a compact invariant transitive subset such there exist $U$ an open neighborhood of $\la$ where the map $Sg_{x|J_x} <0$ where $t\in J_x$ if $(x,t)\in U$. Let $H_s$ be the translation family and suppose that $\la$ is continuously stable by it.

To avoid notation, we denote any $g_x$ with $g$ and $\circ_{i=0}^n g_{h^i(x)}= g^n.$ It is a well known fact that the composition of maps with negative Schawartzian derivative has Schwartzian derivative; in particular,  $Sg^n<0.$

In what follows, if $z(s)$ is the continuation of a point $z\in \la$ we denote with $t_s = \pi_\R(z(s))$. Observe that $\nu(z,s) = t_s -t$. 

From standard arguments in hyperbolic theory, using that orientation preserving hypothesis it follows that if the dynamics in the E-direction is uniformly contracting (expanding), then for $s>0$, $t_s> t$ ($t_s< t$ respectively, cf. proposition \ref{prop A}). In the next lemma we prove a converse to that statement.  

\begin{lema} \label{main l}There exists a positive integer $n_0,$ and  $0<\lambda<1$ such that 
\begin{enumerate}
 \item if $\nu > 0$ for $s>0$ then for any $n>0$ there is $0< m < n_0$ such that $|(g^m)'(g^n(t))|<\lambda;$
 \item if $\nu < 0$ for $s>0$ then for any $n>0$ there is $0< m < n_0$ such that $|(g^{-m})'(g^{-n}(t))|<\lambda.$
\end{enumerate}
\end{lema}

Before giving the the proof of lemma \ref{main l} we show how it implies theorem \ref{sch neg}.

\noi{\em Lemma \ref{main l} implies theorem \ref{sch neg}:} Since $\La$ is transitive, there exists $z\in \la$ such that $\omega(t)=\La$ and $\al(t)=\La.$ If $t_s> t$, then the center direction along $\La$ is contracting; if $t_s< t$, then the center direction $\La$ is expanding. 

\qed

To prove lemma \ref{main l}, first in lemma \ref{bounded interval} we show that if for positive translation the continuation moves to the right (left) then there exists an interval that all its forward (backward) iterate has uniformed bounded length. Later, in lemma \ref{main l 2}, we prove that under the hypothesis of negative Schwartzian, the interval has its  lengths uniformly contracted  and therefore the derivative is contracting. In the last it is used the classical argument that gives uniform distortion for maps with negative Schwarzian derivative.

\begin{lema}\label{bounded interval} There exists $\de>0$ such that 
\begin{enumerate}
 \item if $t_s>t$ then for any $n>0$ and $m>0$ $\ell([g^m(g^n(t)), g^m(g^n(t)_s)])< \de$;
 \item if $t_s<t$ then  for any $n>0$ and $m>0$ $\ell([g^{-m}([g^{-n}(t)]_s), g^{-m}(g^{-n}(t))])< \de$.
\end{enumerate}
 
\end{lema}

\begin{proof}By an inductive argument can be seen that $g^n(t_s)\leq g^n(t)_s$. Using this and that $d(g^n(t)_s, g^n(t))< \de.$ we conclude for the case $t_s>t$.  The proof for $t_s< t$ is similar.
\end{proof}

From the fact that $g$ and its composition has negative Schwartzian derivative bounded away from zero, it follows that there exists $C>0$ and $\de>0$  such that if $\ell(g^n(I))< \de $ then
\begin{eqnarray}\label{dist}
C^{-1}< \frac{|{g^n}'(t)|}{|{g^n}'(r)|} < C,\,\, t, r\in I.
 \end{eqnarray}

Let $\lambda<1$ and let $\ga>0$ be a positive constant such that $C\ga<\lambda<1.$ Let us denote with $I^n_s=[g^n(t), g^n(t)_s]$, i.e., the interval giving by a point and its continuation.

The proof of next lemma is by contradiction: first it is obtained an interval such that the length of its forward iterates remain bounded by above and below; later we show that this implies that the Schwartzian derivative in points in that intervals converges to minus infinity and this will give a contradiction with the hypothesis that the interval is bounded by above.

\begin{lema}\label{main l 2} There exists a positive integer $n_0,$  such that 
\begin{enumerate}
  \item if $t_s>t $ then for any $n>0$ there is $0< m < n_0$ such that $\frac{\ell(g^m(I^n_s))}{\ell(I^n_s)}<\ga;$
 \item if $t_s<t $ then for any $n>0$ there is $0< m < n_0$ such that $\frac{\ell(g^{-m}(I^{-n}_s))}{\ell(I^{-n}_s)}<\ga.$ 
\end{enumerate}
\end{lema}

\begin{proof}
It is enough to prove one of the item. Let us prove the first one. If it is false, for any $n$ there exists $t$ such that for any $m< n$ follows that $\ell(g^m(I_s))>\ga\ell(I_s), $ and by lemma \ref{nulema} it follows that $\ell(g^m(I_s))>\ga. s$ and therefore using the lemma \ref{bounded interval}, taking an accumulation point of the intervals, it follows that there exists an interval $I$ with one of its extreme in $\La$  and $\e>0$ such that 
\begin{eqnarray}\label{int}
\e <\ell(g^m(I))< \de,\,\,\forall \,\,m>0.
 \end{eqnarray}
 From the distortion as in (\ref{dist}) it follows that there is $\be$ such that  for any $r\in I$ follows that 
\begin{eqnarray}\label{bounded der}
\beta^{-1}< |(g^n)'(r)|<\be. 
\end{eqnarray}
Using the formula of the Schwartzian derivative of the composition, $S(f\circ g)= S(f)\circ g (g')^2+ Sg$ and inequality \ref{bounded der} it follows that  
\begin{eqnarray}\label{sf neg} 
S(g^n/I)\to -\infty.
 \end{eqnarray}
 In fact, 
$$S(g^n)= \sum_{i=0}^{n-1} Sg(g^i(r))[(g^i)'(r)]^2\leq S[\sum_{i=0}^{n-1}\be^{-2}]\leq  S\be^{-2}n$$ where $S<0$ is a negative upper bound of $Sg$.  

To conclude the lemma, we will get a contradiction with (\ref{sf neg}). To do that, we analyze the second derivative of $g^n$ and we separate it into cases.

First observe that $(g^n)''_{/I}$ can not be identically zero; if this is the case, $(g^n)'$ would be linear and therefore the $Sg^n=0$, a contradiction.

\noi{\em Case 1. $(g^n)''$  has a zero in the interior of $I$:} since $(g^n)''$ is not identically zero in the interval, there exists either a point $r_0$ and $\e>0$ such that $(g^n)''(r_0)=0$ and  $(g^n)''_{/(r_0, r_0+\e)}>0$ or  a point $r_0$ and $\e>0$ such that $(g^n)''(r_0)=0$ and  $(g^n)''_{/(r_0, r_0+\e)}<0.$ In the former, there exists $r_1$ arbitrarily close to $r_0$ such that $(g^n)'''(r_1)>0;$ in the later, there exists $r_1$ arbitrarily close to $r_0$ such that $(g^n)'''(r_1)<0;$ in both  cases, observe that $\frac{(g^n)'''(r_1)}{(g^n)''(r_1)}>0$, and since $(g^n)''(r_1)$ is small and $(g^n)'(r_1)$ is bounded by below it follows that  $Sg^n(r_1)=\frac{(g^n)'''(r_1)}{(g^n)''(r_1)}-\frac{3}{2}(\frac{(g^n)''(r_1)}{(g^n)'(r_1)})^2$ is arbitrary close to zero, contradicting (\ref{sf neg}).

\noi{\em Case 2. $(g^n)''$  has no  zero in the interior of $I$:} Let us assume that $(g^n)''>0.$ Let $L_n\to -\infty$ such that $Sg^n< L_n.$ We take $0<a< \frac{1}{2}$ and we consider the interval $J_n=\{r: (g^n)''(z)< (-L_n)^a\}.$ This interval $J_n$ is non empty due to the bounded distortion. We claim that
\begin{eqnarray}\label{tercera}
 (g^n)'''_{/J_n}<\frac{1}{2}L_n(g^n)''
\end{eqnarray}
 and in particular $(g^n)'''$ is negative in $J_n$; in fact 
$(g^n)'''< [L_n+ \frac{3}{2}(\frac{(g^n)''}{(g^n)'})^2](g^n)''<[L_n+\frac{3}{2\be^2}(-L_n)^{2a}](g^n)''< L_n[1+\frac{3}{2\be^2}(L_n)^{2a-1}](g^n)'',$ since $(g^n)''>0$ and $a<\frac{1}{2}$ and so  $1+\frac{3}{2\be^2}(L_n)^{2a-1}>\frac{1}{2}$ for large $L_n$ it follows that $(g^n)'''_{/J_n}<\frac{1}{2}L_n(g^n)''.$ In particular, the second derivative on $J_n$ is decreasing and therefore $J_n$ is  a connected interval $(r_n, t_1)$  with $r_n\to t_0$ where $t_0<t_1$ are  the extremal point of $I:$ to prove the last claim observe that if $r\notin J_n$,  then $(g^n)''(r)>-L_n^a$ and since $(g^n)'(y)=(g^n)'(0)+\int_{0}^y (g^n)''(s)ds$ in particular it  follows that  $(g^n)'(r_n)\geq (g^n)'(0)+ r_n (-L_n^a),$ and if $r_n$ does not converge to $t_0$, using that the inequality (\ref{bounded der}) we get a contradiction with (\ref{dist}). Now, using inequality (\ref{tercera}) and that the second derivative is decreasing, it follows that   
\begin{eqnarray*}
(g^n)''(r)&=&(g^n)''(r_n)+ \int^{r_n}_r (g^n)'''(s)ds\\
& \leq& (g^n)''(r_n)+ \frac{1}{2}L_n\int^{r_n}_r (g^n)''(s)ds\\
  & \leq  & (g^n)''(r_n)+ \frac{1}{2}L_n (g^n)''(r_n)(r-r_n)\\
  & \leq & (g^n)''(r_n)[1+\frac{1}{2} L_n(r-r_n)],
\end{eqnarray*}
 so taking $r=t_1$ it follows that  $(g^n)''(t_1)< (g^n)''(r_n)[1+\frac{1}{2} L_n(t_1-r_n)],$
 which is negative since $L_n$ and $t_1-t_0>\frac{s}{2},$ contradicting that $(g^n)''>0.$ 
% \begin{eqnarray*}
%  (f^n)'(y)&=& (f^n)'(x_n)+\int_{t_n}^y (f^n)''(z)dz \\
% & > & (f^n)'(t_n)+(f^n)'(1)(y-t_n)+ \frac{1}{4}(-L_n)[(1-t_n)^2-(1-y)^2],
% \end{eqnarray*}
% since $t_\to t_0$ and $-L_n\to \infty$ using inequality (\ref{bounded der}) we    
%  contradic again (\ref{dist}).

The case that $(g^n)''$ has non zero in the interior of $I$ and  $(g^n)''<0$ is similar: taking $J_n=\{r: (g^n)''(r) > L_n^a\}$ it is concluded that is a connected interval that converges to $(t_0, t_1)$ and  in that interval $(g^n)'''_{/J_n}>\frac{1}{2}L_n(g^n)''$ and repeating as before, it is concluded that $(g^n)''(t_1)$ is positive. 
\end{proof}

\setlength{\parindent}{0 cm}

Javier Correa\\
Universidad Federal Fluminense,\\
Rua M\'ario Santos Braga S/N Valonguinho, 24020-140, Rio de Janeiro, Brazil.\\
\textbf{jacorrea88@gmail.com}\\

Enrique R. Pujals\\
Instituto de Matem\'atica Pura e Aplicada,\\
Estrada Dona Castorina 110, 22460-320,  Rio de Janeiro, Brazil.\\
\textbf{enrique@impa.br}\\


\begin{thebibliography}{99}

\bibitem [CSY] {CSY11}
S. Crovisier, M. Sambarino, D. Yang,
Partial Hyperbolicity and Homoclinic Tangencies.
arXiv:1103.0869.


\bibitem [F] {Fr72}
J. Franks,
Differentiable $\Omega$-stable diffeomorphisms.
\emph{Topology.} 11, (1972), 107-113.

\bibitem [G] {Gu72}
J. Guckenheimer,
Absolutely $\Omega$-stable diffeomorphisms.
\emph{Topology.} 11, (1972), 195-197.

\bibitem [M1] {Ma75}
R. Ma\~n\'e, 
On infinitessimal and Absolute Stability of diffeomorphisms.
\emph{Dynamical systems - Warwick 1974 (Proc. Sympos. Appl.  Topology and Dynamical Systems, Univ. Warwick,  Coventry,  1973/1974; presented to E. C. Zeeman on his fiftieth birthday). Lecture Notes in Math., Springer, Berlin,} 468, (1975), 151-161.

\bibitem[M2]{Ma85}
  R. Ma\~n\'e,
  Hyperbolicity, Sinks and Measure in One Dimensional Dynamics.
  \emph{Commun. Math. Phys.} \textbf{100} (1985), 495-524.

\bibitem [M3] {Ma88}
 R. Ma\~n\'e, 
A proof of the $C^1$ stability conjecture. 
\emph{Inst. Hautes Études Sci. Publ. Math.} No. 66, (1988), 161-210. 

\bibitem [Pa] {Pa88}
J. Palis,
On the $C^1$ $\Omega$-stability conjecture.
\emph{Inst. Hautes Études Sci. Publ. Math.} No. 66, (1988), 211-215.

\bibitem [Pu] {Pu06}
E. Pujals,
On the density of hyperbolicity and homoclinic bifurcations for 3D diffeomorphism in attracting regions.
\emph{ Discrete and Continuous Dynamical Systems.} 16 No. 1, (2006), 179-226.


\bibitem [Pu2] {Pu08} E. Pujals, Some simple questions related to the $C^r$ stability
conjecture. \emph{Nonlinearity.} 21, (2008), 233–237.



\bibitem [PS] {PaSm68}
J. Palis, S. Smale,
Structural Stability Theorems, Global Analysis.
\emph{Proc. Sympos. Pure Math., Vol. XIV, Berkeley, Calif.,} (1968),  223-231, Amer. Math. Soc., Providence, R.I. 

\bibitem [R1] {Ro71}
J. Robbin,
A structural stability theorem.
\emph{Ann. of Math.} 94. No. 3, (1971), 447-493. 

\bibitem [R2] {Ro76}
C. Robinson,
Structural stability of $C^1$ diffeomorphisms.
\emph{Journal of Diff. Eq.} 22(1976), 28-73. 

	
\bibitem[Si]{Si78}	
	D. Singer, Stable orbits and bifurcation of maps of the interval.
 \emph{SIAM J. Appl. Math.}, \textbf{35} (1978), 260-267.

\bibitem [S1] {Sm67}
S. Smale,
Differentiable dynamical systems. 
\emph{Bull. Amer. Math. Soc.} 73, (1967), 747-817.

\bibitem [S2] {Sm70}
S. Smale,
The $\Omega$-stability theorem. 
\emph{Proc. A.M. S. Symp Pure Math} 14 (1970), 289-297.


\bibitem[So] {So12}
J. Solano,
Non-uniform hyperbolicity and existence of absolutely continuous invariant measures.
arXiv:1212.3820.

\bibitem [Ti] {Ti11}
S. Tikhomirov,
H\"older Shadowing on Finite Intervals.
arXiv:1106.4053.

\bibitem [To] {To13}
D. Todorov, 
Generalizations of analogs of theorems of Maizel and Pliss and their application in Shadowing Theory. 
\emph{Discr. Cont. Dyn. Syst.}, ser. A 33, (2013), 4187-4205.

\end{thebibliography}
\end{document}